\documentclass[12pt,A4paper]{amsart}
\usepackage{amsmath,amsthm,amssymb,color}
\usepackage[bookmarksnumbered,colorlinks]{hyperref}
\hypersetup{colorlinks=true,linkcolor=red,anchorcolor=green,citecolor=cyan}

\newtheorem{theorem}{Theorem}[section]

\theoremstyle{remark}
\newtheorem{remark}[theorem]{Remark}
\newtheorem{question}[theorem]{Question}
\numberwithin{equation}{section}
\newcommand{\N}{\mathbb{N}}
\newcommand{\R}{\mathbb{R}}
\newcommand{\C}{\mathbb{C}}

\begin{document}

\title[\tiny Modes of convergence of sequences of holomorphic functions]{Modes of convergence of sequences of holomorphic functions: a linear point of view}

\author[Bernal]{L.~Bernal-Gonz\'alez}
\address[L. Bernal-Gonz\'alez]{\mbox{}\newline \indent Departamento de An\'alisis Matem\'atico \newline \indent Facultad de Matem\'aticas
\newline \indent Instituto de Matem\'aticas de la Universidad de Sevilla (IMUS)
\newline \indent Universidad de Sevilla
\newline \indent Avenida Reina Mercedes s/n, 41012-Sevilla (Spain).}
\email{lbernal@us.es}

\author[Calder\'on]{M.C.~Calder\'on-Moreno}
\address[M.C.~Calder\'on-Moreno]{\mbox{}\newline \indent Departamento de An\'alisis Matem\'atico \newline \indent Facultad de Matem\'aticas
\newline \indent Instituto de Matem\'aticas de la Universidad de Sevilla (IMUS)
\newline \indent Universidad de Sevilla
\newline \indent Avenida Reina Mercedes s/n, 41012-Sevilla (Spain).}
\email{mccm@us.es}

\author[L\'opez-Salazar]{J.~L\'opez-Salazar}
\address[J.~L\'opez-Salazar]{\mbox{}\newline \indent Departamento de Matem\'atica Aplicada a las Tecnolog\'ias  \newline \indent de la Informaci\'on y de las Comunicaciones
\newline \indent Escuela T\'ecnica Superior de Ingenier\'ia y Sistemas de Telecomunicaci\'on
\newline \indent Universidad Polit\'ecnica de Madrid
\newline \indent Nikola Tesla s/n, 28031-Madrid (Spain).}
\email{jeronimo.lopezsalazar@upm.es}

\author[Prado]{J.A.~Prado-Bassas}
\address[J.A.~Prado-Bassas]{\mbox{}\newline \indent Departamento de An\'alisis Matem\'atico
\newline \indent Facultad de Matem\'aticas
\newline \indent Instituto de Matem\'aticas de la Universidad de Sevilla (IMUS)
\newline \indent Universidad de Sevilla
\newline \indent Avenida Reina Mercedes s/n, 41012-Sevilla (Spain).}
\email{bassas@us.es}



\begin{abstract}
In this paper, pointwise convergence, uniform convergence and compact convergence of sequences of holomorphic functions on an open subset of the complex plane are compared from a linear point of view. In fact, it is proved the existence of large linear algebras consisting, except for zero, of sequences of holomorphic functions tending to zero compactly but not uniformly on the open set or of sequences of holomorphic functions tending pointwisely to zero but not compactly. Also dense linear subspaces in an appropriate Fr\'echet space as well as infinite dimensional Banach spaces of sequences converging to zero in the mentioned modes are shown to exist.
\end{abstract}

\maketitle

\section{Introduction}

\quad Along this paper, \,$\Omega$ \,will stand for a nonempty open subset of the complex plane \,$\C$, while \,$H(\Omega)$ \,will represent, as usual, the vector space of all holomorphic functions on \,$\Omega$. When endowed with the topology of the uniform convergence on each compact subset of \,$\Omega$ \,(also called compact convergence), \,$H(\Omega)$ \,becomes a Fr\'echet space. We shall consider three modes of convergence of sequences \,$(f_n)$ \,of holomorphic functions on \,$\Omega$: pointwise convergence, compact convergence, and uniform convergence on the whole domain \,$\Omega$. Plainly, we have the following chain of implications:
\begin{equation*}
  \text{uniform convergence \,$\Longrightarrow$ \,compact convergence \,$\Longrightarrow$ \,pointwise convergence.}
\end{equation*}
Both reverse implications are false but, while it is easy to give examples showing that the first one is not true (take \,$\Omega=\C$ \,and \,$f_n(z)=z/n$), it is not so immediate to provide examples of the failure of the second one. Maybe the reason is that compact convergence and pointwise convergence are very close to each other in some sense. In fact, Osgood's theorem asserts that if \,$(f_n)$ \,converges pointwise in \,$\Omega$, then there is a dense open subset \,$G \subset \Omega$ \,in which \,$(f_n)$ \,converges compactly (see \cite{osgoodpaper} or \cite[p.~151]{Remmert}). Moreover, by the Weierstrass convergence theorem (see, e.g., \cite[p.~176]{Ahlfors}), if \,$f:\Omega\to\C$ \,is the pointwise limit, we have \,$f\in H(G)$. We refer the interested reader to the papers \cite{krantz1,krantz2}, where several kinds of convergence of holomorphic functions are analyzed --even in higher dimensions-- as well as the size of the ``lucky'' open set \,$G$. Osgood's theorem increases the interest in the subject of this paper, that is next described.

\vskip 3pt

In order to study how many --in an appropriate sense-- sequences \,$\mathbf{f}=(f_n)\in H(\Omega)^{\N}$ \,of holomorphic functions in \,$\Omega$ \,enjoy a given mode of convergence but not another one, we define the following three sets of sequences in \,$H(\Omega)$, where, for the sake of normalization, we have selected the zero as the limit function:
\begin{itemize}
  \item $\mathcal{S}_p = \left\{\mathbf{f}=(f_n)\in H(\Omega)^{\N} : \, \lim_{n\to\infty}f_n=0 \, \text{ pointwisely}\right\}$.
  \item $\mathcal{S}_{uc} = \left\{\mathbf{f}=(f_n)\in H(\Omega)^{\N} : \, \lim_{n\to\infty}f_n=0 \, \text{ uniformly on compacta}\right\}$.
  \item $\mathcal{S}_{u} = \left\{\mathbf{f}=(f_n) \in H(\Omega)^{\N} : \, \lim_{n\to\infty}f_n=0 \, \text{ uniformly on $\Omega$}\right\}$.
\end{itemize}
As a special case of the above chain of implications, we have that \,$\mathcal{S}_u \subset \mathcal{S}_{uc} \subset \mathcal{S}_p$.

\vskip 3pt

In this work we are interested in the algebraic and algebraic-topological size of the sets \,$\mathcal{S}_p \setminus \mathcal{S}_{uc}$ \,and \,$\mathcal{S}_{uc} \setminus \mathcal{S}_u$. In fact, it will be proved that both sets contain, except for zero, large linear algebras, as well as dense linear subspaces of an appropriate Fr\'echet space. These dense subspaces can be selected, in addition, so as to contain a prescribed sequence \,{\bf f} \,of each of those sets. Moreover, it will be shown that each of these special families of sequences contain, except for zero, an infinite dimensional Banach space whose norm topology is comparable to the natural one. This will be carried out along the sections 3, 4 and 5, while section 2 will be devoted to recall a number of concepts and results coming from the theory of lineability, as well as to fix the algebraic and topological structure of the set of sequences \,$H(\Omega)^{\N}$. Finally, in section 6 it is established the topological genericity of \,$\mathcal{S}_{uc} \setminus \mathcal{S}_u$.

\section{Algebraic and topological background}\label{background}

\quad Let us consider the vector space \,$H(\Omega)^{\N}$ \,of all sequences \,$(f_n)$ \,of holomorphic functions in \,$\Omega$, where the sum and the multiplication by complex scalars are defined in the usual way. We can also consider the following product of sequences in \,$H(\Omega)$: if \,${\bf f}=(f_n)$ \, and \,${\bf g}=(g_n)$, then
\begin{equation*}
  {\bf f} \cdot {\bf g} = (f_n \cdot g_n).
\end{equation*}
Endowed with this product, it is an elementary fact that \,$H(\Omega)^{\N}$ \,becomes a commutative linear algebra.

\vskip 3pt

Recall that \,$H(\Omega)$ \,has been endowed with the usual compact open topology. If \,$(K_j)$ \,is an exhaustive sequence of compact subsets of \,$\Omega$ \,as in \cite[Theorem 13.3]{Rudin}, then an example of metric \,$d$ \,generating the topology of \,$H(\Omega)$ \,is
\begin{equation*}
  d(f,g) = \displaystyle{\sum_{j=1}^{\infty}}
  \frac{1}{2^j} \cdot \frac{\sup_{z\in K_j} |f(z)-g(z)|}{1+\sup_{z\in K_j} |f(z)-g(z)|}
\end{equation*}
(see, e.g., \cite[pp. 220--221]{Ahlfors}). Now, we endow our space \,$H(\Omega)^{\N}$ \,with the product topology. A complete translation-invariant metric in \,$H(\Omega)^{\N}$ \,is given by
\begin{equation*}
  \widetilde{d}({\bf f},{\bf g}) =
  \sum_{n=1}^{\infty}\frac{1}{2^n} \cdot \frac{d(f_n,g_n)}{1 + d(f_n,g_n)}.
\end{equation*}
This metric induces the product topology on \,$H(\Omega)^{\N}$, making it a Fr\'echet space. Since \,$H(\Omega)$ \,is separable, it follows that \,$H(\Omega)^{\N}$ \,is also separable (see \cite[pp.~370 and 373]{Kothe}).

\vskip 3pt

Now, we turn our attention to lineability. The goal of the modern theory of lineability is to find linear structures inside nonlinear subsets of a vector space. For an extensive study of this theory, the reader is referred to the book \cite{ABPS} (see also \cite{bams,tesisseoane}). Let us recall some of its concepts. A subset \,$A$ \,of a vector space \,$X$ \,is called {\it lineable} whenever there is an infinite dimensional vector subspace of \,$X$ \,that is contained, except for zero, in \,$A$; and \,$A$ \,is said to be {\it algebrable} if it is contained in some linear algebra and there is an infinitely generated algebra contained, except for zero, in \,$A$.  More precisely, if \,$\alpha$ \,is a cardinal number, then \,$A$ \,is said to be {\it $\alpha$-lineable} if there exists a vector subspace \,$V \subset X$ \,such that \,$\dim(V)=\alpha$ \,and \,$V \setminus \{0\} \subset A$. In addition, if \,$A$ \,is contained in some commutative linear algebra, then \,$A$ \,is called {\it strongly $\alpha$-algebrable} if there exists an $\alpha$-generated free algebra \,$\mathcal{M}$ \,with \,$\mathcal{M}\setminus \{0\} \subset A$; that is, there exists a subset \,$B$ \,of cardinality \,$\alpha$ \,with the following property: for any positive integer \,$N$, any nonzero polynomial \,$P$ \,in \,$N$ \,variables without constant term and any distinct elements \,$z_1, \ldots ,z_N \in B$, we have
\begin{equation*}
  P(z_1,\ldots,z_N)\in A \setminus \{0\}.
\end{equation*}
Now, assume that \,$X$ \,is a topological vector space and \,$A \subset X$. Then \,$A$ \,is called {\it dense-lineable} in \,$X$ \,if there is a dense vector subspace \,$V \subset X$ \,such that \,$V\setminus\{0\} \subset A$. If there is a closed infinite dimensional subspace \,$V\subset X$ \,such that \,$V\setminus \{0\} \subset A$, then \,$A$ \,is called {\it spaceable}.

\vskip 3pt

The more accurate notions of pointwise lineable set and infinitely pointwise lineable set have been recently introduced in \cite{PellegrinoRaposo} and \cite{CalderonGerlachPrado}, respectively. Given a cardinal number \,$\alpha$, a subset \,$A$ \,of a vector space \,$X$ \,is {\it pointwise $\alpha$-lineable} if for every \,$x\in A$ \,there exists a vector subspace \,$W_x \subset X$ \,such that \,$\dim(W_x) = \alpha$ \,and \,$x \in W_x \subset A \cup \{0\}$. If \,$X$ \,is a topological vector space, then \,$A$ \,is said to be {\it infinitely pointwise $\alpha$-dense-lineable} in \,$X$ \,if for every \,$x \in A$ \,there is a family \,$\left\{W_k : k \in \N \right\}$ \,such that each \,$W_k$ \,is a dense subspace of \,$X$, $\dim(W_k) = \alpha$, $x \in W_k\subset A \cup \{0\}$, and
\,$W_k \cap W_n = {\rm span}\{x\}$ \,whenever \,$k\neq n$.

\vskip 3pt

We conclude this section with the next result, whose proof can be found in \cite[Theorem 2.3]{CalderonGerlachPrado} and that will be applied in Theorem \ref{Pointwise dense-lineable}:

\begin{theorem}\label{A stronger than B. Pointwise}
Let us suppose that \,$X$ \,is a metrizable separable topological vector space, $\alpha$ \,is an infinite cardinal number, and \,$A$ \,and \,$B$ \,are two
subsets of \,$X$ \,with the following properties:
\begin{enumerate}
  \item $A+B\subset A$.
  \item $A\cap B=\varnothing$.
  \item $A$ \,is pointwise $\alpha$-lineable.
  \item $B$ \,is dense-lineable.
\end{enumerate}
Then the set \,$A$ \,is infinitely pointwise $\alpha$-dense-lineable in \,$X$.
\end{theorem}

\section{Algebrability}

\quad In this section we shall prove the existence of large linear algebras inside our families of sequences \,$\mathcal{S}_p \setminus \mathcal{S}_{uc}$ \,and \,$\mathcal{S}_{uc} \setminus \mathcal{S}_u$. As usual, the symbol \,$\mathfrak{c}$ \,will denote the cardinality of the continuum. Moreover, \,$D(z_0,r)$ \,and \,$\overline{D}(z_0,r)$ \,will stand, respectively, for the open disc and the closed disc in the complex plane with center \,$z_0$ \,and radius \,$r$. By \,$\C_{\infty}$ \,we represent the extended complex plane (or Riemann sphere) \,$\C \cup \{\infty\}$, while \,$\partial A$ \,will denote the boundary in \,$\C$ \,of a
subset \,$A \subset \C$. The interior of a subset \,$A$ \,of a topological space \,$X$ \,will be represented by \,$A^\circ$.

\begin{theorem}\label{Sp-Suc algebrable}
The set \,$\mathcal{S}_p \setminus \mathcal{S}_{uc}$ \,is strongly $\mathfrak{c}$-algebrable.
\end{theorem}

\begin{proof}
Without loss of generality (just practice a translation), we can suppose that \,$0\in \Omega$. Then there is a disc centered at \,$0$ \,contained in \,$\Omega$ \,and thus the value
\begin{equation*}
  R := \sup\left\{x\in\R : x\geq 0 \text{ and } [0,x]\subset \Omega\right\}
\end{equation*}
is strictly positive ($R$ \,could be infinity). Since \,$\Omega$ \,is open, \,$R$ \,cannot belong to \,$\Omega$. Then the set
\begin{equation*}
  G := \Omega \setminus[0,R) = \Omega \setminus[0,R]
\end{equation*}
is open. By \cite[Theorem 13.3]{Rudin}, there is a countable family \,$\left\{K_n:n\in\N\right\}$ \,of compact subsets of \,$G$ \,with the following properties:
\begin{itemize}
  \item $G = \bigcup_{n\in\N} K_n$.
  \item Each \,$K_n$ \,is contained in \,$K_{n+1}^\circ$.
  \item Every compact subset of \,$\Omega\setminus[0,R)$ \,is contained in some \,$K_n$.
  \item For each \,$n\in\N$, every connected component of \,$\C_{\infty}\setminus K_n$ \,contains a connected component of \,$\C_{\infty}\setminus G$.
\end{itemize}
Note also that \,$K_n\cap[0,R) = \varnothing$ \,for every \,$n\in\N$.

\vskip 3pt

We now choose two sequences \,$(s_n)$ \,and \,$(t_n)$ \,of positive real numbers with the following properties:
\begin{itemize}
  \item $0<\cdots<s_3<s_2<s_1<t_1<t_2<t_3<\cdots<R$.
  \item $\lim_{n\to\infty}s_n=0$.
  \item $\lim_{n\to\infty}t_n=R$.
\end{itemize}
For each \,$n\in\N$, the set
\begin{equation*}
  L_n := K_n \cup \{0\} \cup \{s_{n+1}\} \cup [s_n,t_n]
\end{equation*}
is a compact subset of \,$\Omega$. As a consequence, there is \,$r_n>0$ \,small enough to fulfill the next conditions:
\begin{itemize}
  \item $L_n+D(0,r_n)\subset \Omega$.
  \item $\big(K_n+D(0,r_n)\big) \cap \big([0,t_n]+D(0,r_n)\big)=\varnothing$.
  \item $D(0,r_n)\cap D(s_{n+1},r_n)=\varnothing$.
  \item $D(s_{n+1},r_n)\cap D(s_n,r_n)=\varnothing$.
\end{itemize}
Then the set \,$G_n := L_n + D(0,r_n)$ \,is open and it is contained in \,$\Omega$.

\vskip 3pt

Given \,$n\in\N$ \,and \,$c>0$, let \,$g_{c,n}$ \,be the holomorphic function on \,$G_n$ \,defined as
\begin{equation*}
  g_{c,n}(z)=
  \begin{cases}
    0    & \text{if } z\in \big(K_n+D(0,r_n)\big) \cup D(0,r_n) \cup \big([s_n,t_n]+D(0,r_n)\big) \\
    e^{cn} & \text{if } z\in D(s_{n+1},r_n).
  \end{cases}
\end{equation*}
We now consider a set \,$E$ \,determined by choosing one point at each connected component of \,$\C_{\infty}\setminus\Omega$. Then \,$E$ \,contains a point at each connected component of \,$\C_{\infty}\setminus G$ \,and, as a consequence, at each connected component of \,$\C_{\infty}\setminus K_n$ \,and of \,$\C_{\infty}\setminus L_n$ \,for every \,$n\in\N$. By the Runge approximation theorem (see \cite[Theorem 13.6]{Rudin}), there is a rational function \,$f_{c,n}$ \,whose poles lie in \,$E\subset \C_{\infty} \setminus \Omega$ \,such that
\begin{equation*}
  \left|f_{c,n}(z)-g_{c,n}(z)\right| < e^{-n^2}
\end{equation*}
for all \,$z\in L_n$. In particular, \,$f_{c,n}$ \,is holomorphic on \,$\Omega$.

\vskip 3pt

Inspired by \cite[Proposition 7]{Balcerzak2013} and \cite[Proposition 2.3]{Bernal2018}, we consider an algebraic basis \,$\mathcal{H}$ \,for \,$\R$ \,over the field \,$\mathbb{Q}$ \,contained in the interval \,$(0,1]$. It is known that the cardinality of such a basis is \,$\mathfrak{c}$ \,(see \cite[Theorem 4.2.3]{Kuczma}). We will prove that the set of sequences
\begin{equation*}
  \left\{{\bf f}_h = \big(f_{h,n}\big) : h\in\mathcal{H}\right\}
\end{equation*}
generates a free algebra contained in \,$\mathcal{S}_p\setminus\mathcal{S}_{uc}$. In order to do that, we must check that if \,$h_1,\ldots,h_N$ \,are distinct elements of \,$\mathcal{H}$ \,and \,$P$ \,is a non-zero polynomial in \,$N$ \,variables without constant term, then the sequence
\begin{equation*}
  P\left({\bf f}_{h_1},\ldots,{\bf f}_{h_N}\right)
  = \big(P \left(f_{h_1,n},\ldots,f_{h_N,n} \right) \big)
\end{equation*}
is not zero and converges to zero at every point but the convergence is not uniform on compacta.

\vskip 3pt

First of all, if \,$z \in \Omega$, then either \,$z\in[0,R)$ \,or \,$z\in\Omega\setminus[0,R) = G$. If \,$z\in[0,R)$, then there is \,$n_0\in\N$ \,such that \,$z\in \{0\}\cup[s_n,t_n]$ \,for all \,$n\geq n_0$. If \,$z\in G = \bigcup_{n\in\N}K_n$, then there is \,$n_0\in\N$ \,such that \,$z\in K_{n_0}\subset K_n$ \,for all \,$n\geq n_0$. In both cases, if \,$c>0$ \,and \,$n\geq n_0$, then
\begin{equation*}
  |f_{c,n}(z)| \leq |f_{c,n}(z)-g_{c,n}(z)|+|g_{c,n}(z)| < e^{-n^2}+0 = e^{-n^2} \longrightarrow 0
  \text{ \ as } n\to\infty.
\end{equation*}
That is, \,$\lim_{n\to\infty}f_{c,n}(z)=0$ \,for every \,$z\in\Omega$ \,and every \,$c>0$. From the continuity of \,$P$ \,and the fact that \,$P(0,0,\dots,0)=0$, it follows that
\begin{equation*}
  \lim_{n\to\infty} P\left(f_{h_1,n}(z),\ldots,f_{h_N,n}(z)\right)=0
\end{equation*}
for every \,$z\in\Omega$. That is, the sequence \,$\big(P\left(f_{h_1,n},\ldots,f_{h_N,n}\right)\big)$ \,belongs to \,$\mathcal{S}_p$.

\vskip 3pt

Next we will prove that \,$\big(P\left(f_{h_1,n},\ldots,f_{h_N,n}\right)\big)$ \,does not converge to zero uniformly on the compact set \,$\{0\}\cup\left\{s_{n+1}:n\in\N\right\}$. Since \,$P$ \,does not contain constant term, it can be written as
\begin{equation*}
  P(z_1,\ldots,z_N) = \sum_{j=1}^{k}\lambda_j z_1^{\alpha_{j1}}\cdots z_N^{\alpha_{jN}},
\end{equation*}
where \,$k\in\N$, $\lambda_j \in \C\setminus\{0\}$, $\alpha_{jl} \in \N\cup\{0\}$ \,for all \,$1\leq j\leq k$ \,and \,$1\leq l\leq N$, $\alpha_{j1} + \cdots + \alpha_{jN} > 0$ \,for all \,$j$, and the vectors of exponents \,$(\alpha_{11},\ldots,\alpha_{1N}),\ldots,(\alpha_{k1},\ldots,\alpha_{kN})$ \,are all different. Let
\begin{equation*}
  c_j := \alpha_{j1}h_1+\cdots+\alpha_{jN}h_N
\end{equation*}
for each \,$j\in\{1,\ldots,k\}$. Since \,$\mathcal{H}$ \,is an algebraic basis over \,$\mathbb{Q}$, it follows that \,$c_1,\ldots,c_k$ \,are distinct positive numbers, so we can assume that \,$0<c_1<\cdots<c_k$. Therefore,
\begin{align}\label{Sp-Suc algebrable.1}
  \lim_{n\to\infty} &\left|P\left(g_{h_1,n}(s_{n+1}),\ldots,g_{h_N,n}(s_{n+1})\right)\right|
  = \lim_{n\to\infty} \left|P\left(e^{h_1 n},\ldots,e^{h_N n}\right)\right| \nonumber \\
  &\qquad =\lim_{n\to\infty} \left|\sum_{j=1}^{k}\lambda_j \left(e^{h_1 n}\right)^{\alpha_{j1}} \cdots \left(e^{h_N n}\right)^{\alpha_{jN}} \right|
  = \lim_{n\to\infty} \left|\sum_{j=1}^{k}\lambda_j e^{c_jn}\right| \nonumber \\
  &\qquad =\lim_{n\to\infty} e^{c_kn}\left|\lambda_k + \sum_{j=1}^{k-1}\lambda_j e^{(c_j-c_k)n}\right|
  = +\infty.
\end{align}

\vskip 2pt

Let us endow the space \,$\C^N$ \,with the supremum norm. Fixed \,$n\in\N$, we define the vectors \,$x$ \,and \,$y$ \,as follows:
\begin{equation*}
  x = \left(f_{h_1,n}(s_{n+1}),\ldots,f_{h_N,n}(s_{n+1})\right),
\end{equation*}
\begin{equation*}
  y = \left(g_{h_1,n}(s_{n+1}),\ldots,g_{h_N,n}(s_{n+1})\right).
\end{equation*}
Since \,$h_1,\ldots,h_N\in(0,1]$, we have
\begin{equation*}
  \Vert y\Vert = \max_{1\leq j\leq N} \left|g_{h_j,n}(s_{n+1})\right|
  = \max_{1\leq j\leq N} e^{h_j n} \leq e^n.
\end{equation*}
Moreover,
\begin{equation*}
  \Vert x-y \Vert = \max_{1\leq j\leq N} \left|f_{h_j,n}(s_{n+1})-g_{h_j,n}(s_{n+1})\right|
  < e^{-n^2}.
\end{equation*}
If \,$M$ \,is the degree of the polynomial \,$P$, we write \,$P=\sum_{m=1}^{M}P_m$, where each \,$P_m$ \,(with \,$1\leq m\leq M$) is an $m$-homogeneous polynomial (recall that it was assumed that \,$P$ \,did not have constant term). For each \,$m\in\{1,\ldots,M\}$, let
\begin{equation*}
  A_m: \underbrace{\C^N\times\cdots\times\C^N}_{m \text{ times}} \to \C
\end{equation*}
be the symmetric $m$-linear mapping associated to \,$P_m$ (see \cite[Corollary 2.3]{Mujica}). The continuity of \,$P_m$ \,and \,$A_m$ \,implies that \,$\Vert A_m\Vert<\infty$ (see \cite[Proposition 1.2]{Mujica}). Moreover, the Newton binomial formula implies that
\begin{equation*}
  P_m(x) = A_m(\underbrace{x,\ldots,x}_{m\text{ times}})
  = \sum_{j=0}^{m} \binom{m}{j}
  A_m(\underbrace{y,\ldots,y}_{j\text{ times}},\underbrace{x-y,\ldots,x-y}_{m-j \text{ times}})
\end{equation*}
(see \cite[Corollary 1.9]{Mujica}). Hence,
\begin{align*}
  |P_m(x)-P_m(y)|
  &\leq \sum_{j=0}^{m-1} \binom{m}{j}
  |A_m(\underbrace{y,\ldots,y}_{j\text{ times}},\underbrace{x-y,\ldots,x-y}_{m-j \text{ times}})| \\
  &\leq \sum_{j=0}^{m-1} \binom{m}{j} \Vert A_m\Vert \cdot \Vert y\Vert^j
  \cdot \Vert x-y\Vert^{m-j} \\
  &\leq \sum_{j=0}^{m-1} \binom{m}{j} \Vert A_m\Vert \cdot e^{nj} \cdot e^{-n^2(m-j)}
  \leq m \cdot m! \Vert A_m\Vert \cdot e^{n(m-1)-n^2}.
\end{align*}
As this holds for all \,$n\in\N$ \,and \,$m\in\{1,\ldots,M\}$, we have that
\begin{align}\label{Sp-Suc algebrable.2}
  &\left|P\left(f_{h_1,n}(s_{n+1}),\ldots,f_{h_N,n}(s_{n+1})\right) - P\left(g_{h_1,n}(s_{n+1}),\ldots,g_{h_N,n}(s_{n+1})\right) \right|
  = \left|P(x)-P(y)\right| \nonumber \\
  &\hspace{1cm} = \left|\sum_{m=1}^{M}(P_m(x)-P_m(y))\right|
  \leq \sum_{m=1}^{M} m \cdot m! \Vert A_m\Vert \cdot e^{n(m-1)-n^2}
  \longrightarrow 0 \quad \text{as } n\to\infty.
\end{align}
The limits in \eqref{Sp-Suc algebrable.1} and \eqref{Sp-Suc algebrable.2} imply that
\begin{equation*}
  \lim_{n\to\infty}\left|P\left(f_{h_1,n}(s_{n+1}),\ldots,f_{h_N,n}(s_{n+1})\right)\right|
  = +\infty.
\end{equation*}
This proves that the sequence \,$\big(P\left(f_{h_1,n},\ldots,f_{h_N,n}\right)\big)$ \,does not converge uniformly to zero on the compact set \,$\{0\}\cup\left\{s_{n+1}:n\in\N\right\}$. In particular, \,$\big(P\left(f_{h_1,n},\ldots,f_{h_N,n}\right)\big)$ \,is not zero and belongs to \,$\mathcal{S}_p\setminus \mathcal{S}_{uc}$. That concludes the proof.
\end{proof}

\vskip 3pt

\begin{theorem}\label{Suc-Su algebrable}
The set \,$\mathcal{S}_{uc}\setminus \mathcal{S}_u$ \,is strongly $\mathfrak{c}$-algebrable.
\end{theorem}

\begin{proof}
As in the proof of Theorem \ref{Sp-Suc algebrable}, we can assume without loss of generality that \,$0 \in \Omega$ \,and then we define the extended number \,$R$ \,as in that proof. By \cite[Theorem 13.3]{Rudin}, there is a countable family \,$\left\{K_n:n\in\N\right\}$ \,of compact subsets of \,$\Omega$ \,such that \,$\Omega=\bigcup_{n\in\N}K_n$, each \,$K_n$ \,is contained in \,$K_{n+1}^\circ$, every compact subset of \,$\Omega$ \,is contained in some \,$K_n$, and every connected component of \,$\C_{\infty}\setminus K_n$ \,contains a connected component of \,$\C_{\infty}\setminus\Omega$ \,for each \,$n\in\N$.

Since \,$R$ \,does not belong to \,$\Omega$, we get that \,$R\notin K_n$ \,for any \,$n\in\N$. Therefore, there exists a sequence \,$(t_n)$ \,of positive real numbers with the following properties:
\begin{itemize}
  \item $0<t_1<t_2<t_3<\cdots<R$.
  \item $\lim_{n\to\infty}t_n=R$.
  \item $t_n\notin K_n$.
\end{itemize}
By the definition of \,$R$, we have that every \,$t_n$ \,belongs to \,$\Omega$. For each \,$n \in \N$, there is \,$r_n >0$ \,small enough to fulfill the next conditions:
\begin{itemize}
  \item $K_n+D(0,r_n)\subset \Omega$.
  \item $D(t_n,r_n)\subset \Omega$.
  \item $\big(K_n+D(0,r_n)\big) \cap D(t_n,r_n)=\varnothing$.
\end{itemize}
Then the set \,$G_n := \big(K_n+D(0,r_n)\big) \cup D(t_n,r_n)$ \,is open and it is contained in \,$\Omega$.

\vskip 3pt

Given \,$n\in\N$ \,and \,$c>0$, we define the function \,$g_{c,n} : G_n \to \C$ \,as
\begin{equation*}
  g_{c,n}(z) := \begin{cases}
                0          & \text{if } z\in K_n+D(0,r_n) \\
                e^{cn} & \text{if } z\in D(t_n,r_n).
                \end{cases}
\end{equation*}
We now consider a set \,$E$ \,defined by choosing one point at each connected component of \,$\C_{\infty} \setminus \Omega$. Then \,$E$ \,contains a point at each connected component of \,$\C_{\infty}\setminus K_n$ \,and, as a consequence, at each connected component of \,$\C_{\infty}\setminus \left(K_n\cup\{t_n\}\right)$ \,for every \,$n\in\N$. By the Runge approximation theorem, there is a rational function \,$f_{c,n}$ \,whose poles lie in \,$E \subset \C_{\infty}\setminus\Omega$ \,such that
\begin{equation*}
  \left|f_{c,n}(z) - g_{c,n}(z)\right| < e^{-n^2}
\end{equation*}
for all \,$z \in K_n\cup\{t_n\}$. In particular, \,$f_{c,n}$ \,is holomorphic in \,$\Omega$.

\vskip 3pt

Let \,$\mathcal{H}$ \,be an algebraic basis for \,$\R$ \,over the field \,$\mathbb{Q}$ \,contained in the interval \,$(0,1]$. Our goal is to show that the set of sequences
\begin{equation*}
  \left\{{\bf f}_h = \big(f_{h,n}\big) : h\in\mathcal{H}\right\},
\end{equation*}
whose cardinality is plainly $\mathfrak{c}$, generates a free algebra contained in \,$\mathcal{S}_{uc}\setminus\mathcal{S}_u$. Let us suppose that \,$h_1,\ldots,h_N$ \,are distinct elements of \,$\mathcal{H}$ \,and that \,$P$ \,is a non-zero polynomial in \,$N$ \,variables without constant term. The polynomial \,$P$ \,can be written as
\begin{equation*}
  P(z_1,\ldots,z_N) = \sum_{j=1}^{k}\lambda_j z_1^{\alpha_{j1}}\cdots z_N^{\alpha_{jN}},
\end{equation*}
where \,$k\in\N$, $\lambda_j\in\C\setminus\{0\}$, $\alpha_{jl}\in\N\cup\{0\}$ \,for all \,$1\leq j\leq k$ \,and \,$1\leq l\leq N$, $\alpha_{j1}+\cdots+\alpha_{jN}>0$ \,for all \,$j$, and the vectors of exponents \,$(\alpha_{11},\ldots,\alpha_{1N}),\ldots,(\alpha_{k1},\ldots,\alpha_{kN})$ \,are different. If \,$K$ \,is any compact subset of \,$\Omega$, there exists \,$n_0\in\N$ \,such that \,$K\subset K_n$ \,for all \,$n\geq n_0$. Hence, if \,$n\geq n_0$, we have
\begin{align*}
  \sup_{z\in K}\left|P\left(f_{h_1,n}(z),\ldots,f_{h_N,n}(z)\right)\right|
  &\leq \sup_{z\in K_n}
  \sum_{j=1}^{k}|\lambda_j| \cdot |f_{h_1,n}(z)|^{\alpha_{j1}}\cdots |f_{h_N,n}(z)|^{\alpha_{jN}} \\
  &\leq \sup_{z\in K_n} \sum_{j=1}^{k}|\lambda_j| \cdot \big(|f_{h_1,n}(z)-g_{h_1,n}(z)|+|g_{h_1,n}(z)|\big)^{\alpha_{j1}} \cdots \\
  &\qquad\qquad \cdots \big(|f_{h_N,n}(z)-g_{h_N,n}(z)|+|g_{h_N,n}(z)|\big)^{\alpha_{jN}} \\
  &< \sum_{j=1}^{k}|\lambda_j| \cdot e^{-n^2\alpha_{j1}}\cdots e^{-n^2\alpha_{jN}}
  \longrightarrow 0 \quad \text{as } n\to\infty.
\end{align*}
This proves that the sequence \,$P\left({\bf f}_{h_1},\ldots,{\bf f}_{h_N}\right) =\big(P\left(f_{h_1,n},\ldots,f_{h_N,n}\right)\big)$ \,converges to zero uniformly on every compact subset of \,$\Omega$.

\vskip 3pt

Now, let
\begin{equation*}
  c_j := \alpha_{j1}h_1+\cdots+\alpha_{jN}h_N
\end{equation*}
for each \,$j\in\{1,\ldots,k\}$. Since \,$\mathcal{H}$ \,is an algebraic basis over \,$\mathbb{Q}$, it follows that \,$c_1,\ldots,c_k$ \,are distinct positive numbers and we can assume that \,$0 < c_1 < \cdots < c_k$. Therefore,
\begin{align}\label{Suc-Su algebrable.1}
  \lim_{n\to\infty} & \left|P\left(g_{h_1,n}(t_n),\ldots,g_{h_N,n}(t_n)\right)\right|
  = \lim_{n\to\infty} \left|P\left(e^{h_1n},\ldots,e^{h_Nn}\right)\right| \nonumber \\
  &= \lim_{n\to\infty} \left|\sum_{j=1}^{k}\lambda_j \left(e^{h_1n}\right)^{\alpha_{j1}} \cdots \left(e^{h_Nn}\right)^{\alpha_{jN}}\right|
  = \lim_{n\to\infty} \left|\sum_{j=1}^{k}\lambda_j e^{c_jn}\right| \nonumber \\
  &= \lim_{n\to\infty} e^{c_kn}\left|\lambda_k + \sum_{j=1}^{k-1}\lambda_j e^{(c_j-c_k)n}\right|
  = +\infty.
\end{align}

\vskip 3pt

The space \,$\C^N$ \,will be again endowed with the supremum norm. Fixed \,$n\in\N$, we define the vectors \,$x$ \,and \,$y$ \,as follows:
\begin{equation*}
  x = \left(f_{h_1,n}(t_n),\ldots,f_{h_N,n}(t_n)\right),
\end{equation*}
\begin{equation*}
  y = \left(g_{h_1,n}(t_n),\ldots,g_{h_N,n}(t_n)\right).
\end{equation*}
Then
\begin{equation*}
  \Vert y\Vert = \max_{1\leq j\leq N} e^{h_j n} \leq e^n
\end{equation*}
and
\begin{equation*}
  \Vert x-y \Vert = \max_{1\leq j\leq N} \left|f_{h_j,n}(t_n)-g_{h_j,n}(t_n)\right| < e^{-n^2}.
\end{equation*}
Set \,$P=\sum_{m=1}^{M}P_m$, where \,$M$ \,is the degree of \,$P$ \,and each \,$P_m$ \,(with \,$1\leq m\leq M$) is an $m$-homogeneous polynomial. For each \,$m\in\{1,\ldots,M\}$, let \,$A_m$ \,denote the symmetric $m$-linear mapping associated to \,$P_m$. As in the proof of Theorem \ref{Sp-Suc algebrable}, we have
\begin{align}\label{Suc-Su algebrable.2}
  &\left|P\left(f_{h_1,n}(t_n),\ldots,f_{h_N,n}(t_n)\right) - P\left(g_{h_1,n}(t_n),\ldots,g_{h_N,n}(t_n)\right) \right|
  = \left|P(x)-P(y)\right| \nonumber \\
  &\hspace{1cm} = \left|\sum_{m=1}^{M}(P_m(x)-P_m(y))\right|
  \leq \sum_{m=1}^{M}m \cdot m! \Vert A_m\Vert \cdot e^{n(m-1)-n^2}
  \longrightarrow 0 \quad \text{as } n\to\infty.
\end{align}
The limits in \eqref{Suc-Su algebrable.1} and \eqref{Suc-Su algebrable.2} imply that
\begin{equation*}
  \lim_{n\to\infty}\left|P\left(f_{h_1,n}(t_n),\ldots,f_{h_N,n}(t_n)\right)\right|
  = +\infty.
\end{equation*}
This proves that the sequence \,$\big(P\left(f_{h_1,n},\ldots,f_{h_N,n}\right)\big)$ \,does not converge to zero uniformly on \,$\Omega$. In particular, the sequence \,$\big(P\left(f_{h_1,n},\ldots,f_{h_N,n}\right)\big)$ \,cannot be zero and belongs to \,$\mathcal{S}_{uc}\setminus \mathcal{S}_u$. That concludes the proof.
\end{proof}

\section{Pointwise lineability}

\quad In this section, we will show that for each sequence \,${\bf f}$ \,in \,$\mathcal{S}_p\setminus\mathcal{S}_{uc}$ \,and in \,$\mathcal{S}_{uc}\setminus\mathcal{S}_{u}$ \,it is possible to obtain large spaces containing \,${\bf f}$ \,and contained in \,$\mathcal{S}_p\setminus\mathcal{S}_{uc}$ \,and \,$\mathcal{S}_{uc}\setminus\mathcal{S}_{u}$, respectively.

\begin{theorem}\label{Sp-Suc lineabilidad puntual}
The set \,$\mathcal{S}_p\setminus\mathcal{S}_{uc}$ \,is pointwise $\mathfrak{c}$-lineable.
\end{theorem}

\begin{proof}
Let us fix a sequence \,${\bf f}=(f_n)\in\mathcal{S}_p \setminus\mathcal{S}_{uc}$. For each \,$c\in\R$ \,we define \,$\varphi_c(z)=e^{cz}$.
Let \,$W$ \,be the following vector subspace of \,$H(\C )^{\N}$:
\begin{equation*}
  W = {\rm span} \left\{(\varphi_c\cdot f_n) : \, c \in [0,+\infty ) \right\}.
\end{equation*}
Choosing \,$c=0$, we obtain that \,${\bf f} \in W$. Our task is to show that \,$\dim (W) = \mathfrak{c}$ \,and that \,$W \setminus\{0\}$ \,is contained in \,$\mathcal{S}_p\setminus\mathcal{S}_{uc}$.

\vskip 3pt

With this aim, let \,$m\in\N$, \,$\lambda_1,\ldots,\lambda_m \in \C \setminus \{0\}$, and \,$c_1,\ldots,c_m$ \,distinct non-negative real numbers.
We can suppose that \,$0 \leq c_1 < \cdots < c_m$. Since \,$\lim_{n\to\infty} f_n(z) = 0$ \,for every \,$z \in \Omega$, it is plain that the sequence \,$\left(\sum_{j=1}^{m}\lambda_j \varphi_{c_j}\cdot f_n\right)$ \,belongs to \,$\mathcal{S}_p$. Moreover, since \,$(f_n)$ \,does not converge to zero uniformly on compacta, there exist a compact set \,$K \subset \Omega$, a positive number \,$\varepsilon$, and a subsequence \,$(f_{n_k})$ \,of \,${\bf f}$ \,such that
\begin{equation*}
  \sup_{z\in K}|f_{n_k}(z)| > \varepsilon
\end{equation*}
for all \,$k\in\N$. Now, for each \,$k\in\N$ \,there is \,$z_k \in K$ \,such that \,$|f_{n_k}(z_k)|>\varepsilon$. Since \,$K$ \,is compact, there is a subsequence \,$(z_{k_\ell})$ \,converging to a point \,$z_0\in K$. Fix an \,$R>0$ \,such that \,$K \subset D(0,R)$.

\vskip 3pt

The function \,$g(z)=\sum_{j=1}^{m}\lambda_j e^{c_jz}$ \,is not identically zero on \,$\C$. Indeed, if \,$c_m>0$, then
\begin{equation*}
  \lim_{\underset{x\in\R}{x\to+\infty}}\left|\sum_{j=1}^{m}\lambda_j e^{c_jx}\right|
  = \lim_{\underset{x\in\R}{x\to+\infty}}
  e^{c_m x}\left|\lambda_m + \sum_{j=1}^{m-1}\frac{\lambda_j}{e^{(c_m-c_j)x}}\right|
  = +\infty,
\end{equation*}
while if \,$c_m=0$, then \,$m=1$ \,and \,$g(z)=\lambda_1\neq 0$ \,for all \,$z\in\C$. In both cases, by the identity principle, \,$g$ \,has at most a finite amount of zeros on \,$D(0,R)$. Then there exists \,$r>0$ \,such that \,$\overline{D}(z_0,r)\subset\Omega$ \,and \,$g(z)\neq 0$ \,if \,$|z-z_0|=r$. Then \,$\min_{|z - z_0|=r}|g(z)|>0$. Moreover, since \,$\lim_{\ell\to\infty}z_{k_\ell}=z_0$, there exists \,$\ell_0\in\N$ \,such that \,$z_{k_\ell}\in D(z_0,r)$ \,for all \,$\ell\geq \ell_0$. By the maximum modulus theorem, for each \,$\ell\geq \ell_0$ \,we have that
\begin{align*}
  \sup_{|z-z_0|=r}\left|\sum_{j=1}^{m}\lambda_j e^{c_jz}\cdot f_{n_{k_\ell}}(z)\right|
  &\geq \sup_{|z-z_0|=r}|f_{n_{k_\ell}}(z)| \cdot \min_{|z-z_0|=r}|g(z)| \\
  &\geq \sup_{\ell \geq \ell_0}|f_{n_{k_\ell}}(z_{k_\ell})| \cdot \min_{|z-z_0|=r}|g(z)| \\
  &> \varepsilon \cdot \min_{|z-z_0|=r}|g(z)| > 0.
\end{align*}
Therefore, the sequence \,$\left(\sum_{j=1}^{m}\lambda_j \varphi_{c_j}\cdot f_n\right)$ \,does not converge to zero uniformly on \,$\left\{z\in\C:|z-z_0| = r\right\}$, which entails \,$\left(\sum_{j=1}^{m} \lambda_j \varphi_{c_j} \cdot f_n\right)\notin\mathcal{S}_{uc}$. In particular, that sequence cannot be zero, which in turn shows that the set \,$\left\{(\varphi_c\cdot f_n):c\geq 0\right\}$ \,is linearly independent. Thus, \,$W$ \,is a vector space such that \,$\dim(W) = \mathfrak{c}$ \,and
\begin{equation*}
  {\bf f} \in W \setminus\{0\} \subset \mathcal{S}_p \setminus\mathcal{S}_{uc}.
\end{equation*}
This proves the desired pointwise $\mathfrak{c}$-lineability of \,$\mathcal{S}_p\setminus\mathcal{S}_{uc}$.
\end{proof}

\vskip 3pt

\begin{theorem}\label{Suc-Su lineabilidad puntual}
The set \,$\mathcal{S}_{uc}\setminus\mathcal{S}_u$ \,is pointwise $\mathfrak{c}$-lineable.
\end{theorem}

\begin{proof}
Let us fix a sequence \,${\bf f} = (f_n) \in \mathcal{S}_{uc} \setminus \mathcal{S}_u$. Since \,$(f_n)$ \,does not converge to zero uniformly on \,$\Omega$, there is a subsequence \,$(f_{n_k})$ \,and a positive number \,$\varepsilon$ \,such that
\begin{equation*}
  \sup_{z\in\Omega}|f_{n_k}(z)| > \varepsilon
\end{equation*}
for all \,$k\in\N$. For each \,$k\in\N$ \,there is \,$z_k\in\Omega$ \,such that \,$|f_{n_k}(z_{n_k})|>\varepsilon$. The sequence \,$\left\{z_k:k\in\N\right\}$ \,cannot be relatively compact in \,$\Omega$ \,because in such a case \,$(f_n)$ \,would converge to zero uniformly on the compact set \,$\overline{\left\{z_k:k\in\N\right\}}$, which is clearly false. Thus, there are two possibilities:
\begin{itemize}
  \item The sequence \,$\left\{z_k:k\in\N\right\}$ \,is unbounded.
  \item There is a point \,$z_0 \in \partial \Omega$ \,such that \,$z_0 \in \overline{\left\{z_k: \, k \in \N \right\}}$.
\end{itemize}

\vskip 2pt

We shall first suppose that \,$\left\{z_k:k\in\N\right\}$ \,is unbounded. Then there is a subsequence, that we will also label \,$(z_k)$, such that
\,$\lim_{k\to\infty}|z_k|=+\infty$. For each \,$k\in\N$ \,there is \,$\alpha_k\in\C$ \,such that \,$|\alpha_k|=1$ \,and \,$\alpha_k z_k=|z_k|$. Given \,$c\geq 0$ \,and \,$k\in\N$, we define the function \,$\varphi_{c,n}$ \,on \,$\C$ \,as follows:
\begin{equation*}
  \varphi_{c,n}(z)=\begin{cases}
                     e^{c \alpha_k z} & \text{if $n=n_k$ for some } k\in\N \\
                     e^{c z}          & \text{if } n\notin\{n_k:k\in\N\}.
                   \end{cases}
\end{equation*}
Let \,$W$ \,be the following vector subspace of \,$H(\C)^{\N}$:
\begin{equation*}
  W = {\rm span} \left\{(\varphi_{c,n}\cdot f_n) : c\in[0,+\infty)\right\}.
\end{equation*}
Choosing \,$c=0$, we obtain that \,${\bf f}=(f_n)\in W$. We will prove that \,$\dim(W)=\mathfrak{c}$ \,and that \,$W\setminus\{0\}$ \,is contained in \,$\mathcal{S}_{uc}\setminus\mathcal{S}_u$.

\vskip 3pt

Let \,$m\in\N$, \,$\lambda_1,\ldots,\lambda_m\in\C\setminus\{0\}$, and \,$0\leq c_1<\cdots<c_m$. For each compact subset \,$K\subset\Omega$ \,there is \,$R>0$ \,such that \,$K\subset \overline{D}(0,R)$. Note that
\begin{equation*}
  \sup_{z\in\overline{D}(0,R)}\left|e^{c\alpha_k z}\right|
  = \sup_{z\in\overline{D}(0,R)}\left|e^{cz}\right|
  = e^{cR}
\end{equation*}
for every \,$c\geq 0$ \,and every \,$k\in\N$. Therefore,
\begin{align*}
  \sup_{z\in K}\left|\sum_{j=1}^{m}\lambda_j \varphi_{c_j,n}(z)\cdot f_n(z)\right|
  &\leq \sum_{j=1}^{m}|\lambda_j| \cdot \sup_{z\in \overline{D}(0,R)}\left|\varphi_{c_j,n}(z)\right|
  \cdot \sup_{z\in K}|f_n(z)| \\
  &= \sum_{j=1}^{m}|\lambda_j| \cdot e^{c_j R} \cdot \sup_{z\in K}|f_n(z)| \longrightarrow 0
  \text{ as } n\to\infty
\end{align*}
because \,$(f_n)$ \,converges to zero uniformly on compacta. This proves that the sequence \,$\left(\sum_{j=1}^{m}\lambda_j \varphi_{c_j,n}\cdot f_n\right)$ \,belongs to \,$\mathcal{S}_{uc}$.

\vskip 3pt

Now we will prove that \,$\left(\sum_{j=1}^{m}\lambda_j \varphi_{c_j,n}\cdot f_n\right)$ \,does not converge to zero uniformly on \,$\left\{z_k:k\in\N\right\}$. Since \,$c_1<\cdots<c_m$ \,and \,$\lim_{k\to\infty}|z_k|=+\infty$, there is \,$k_0\in\N$ \,such that
\begin{equation*}
  \left|\sum_{j=1}^{m-1}\lambda_j e^{(c_j-c_m)|z_k|}\right| \leq \frac{|\lambda_m|}{2}
\end{equation*}
for all \,$k\geq k_0$. Then
\begin{align*}
  \sup_{k\geq k_0}\left|\sum_{j=1}^{m}\lambda_j \varphi_{c_j,n_k}(z_k)\cdot f_{n_k}(z_k)\right|
  &\geq \sup_{k\geq k_0} \varepsilon\left|\sum_{j=1}^{m}\lambda_j e^{c_j \alpha_k z_k}\right|
  = \sup_{k\geq k_0}\varepsilon\left|\sum_{j=1}^{m}\lambda_j e^{c_j|z_k|}\right| \\
  &= \sup_{k\geq k_0} \varepsilon \cdot e^{c_m|z_k|}
  \left|\lambda_m + \sum_{j=1}^{m-1} \lambda_j e^{(c_j-c_m)|z_k|}\right| \\
  &\geq \sup_{k\geq k_0} \varepsilon \cdot e^{c_m|z_k|} \cdot \frac{|\lambda_m|}{2}
  \geq \varepsilon \cdot \frac{|\lambda_m|}{2} > 0.
\end{align*}
This proves that \,$\left(\sum_{j=1}^{m}\lambda_j \varphi_{c_j,n}\cdot f_n\right)$ \,does not converge to zero uniformly on \,$\left\{z_k:k\geq k_0\right\}$. In particular, \,$\left(\sum_{j=1}^{m}\lambda_j \varphi_{c_j,n}\cdot f_n\right)$ \,is not the zero sequence, which in turn implies that the set \,$\left\{(\varphi_{c,n}\cdot f_n) : c\in[0,+\infty)\right\}$ \,is linearly independent, so \,$\dim(W)=\mathfrak{c}$. That completes the proof in the case that \,$\left\{z_k:k\in\N\right\}$ \,is unbounded.

\vskip 3pt

Next, we assume that there is a point \,$z_0\in\partial\Omega$ \,such that \,$z_0\in\overline{\left\{z_k:k\in\N\right\}}$. Then there is a subsequence, that we will continue labeling \,$(z_k)$, such that \,$\lim_{k\to\infty}z_k=z_0$. For each \,$k\in\N$, there is \,$\alpha_k\in\C$ \,such that \,$|\alpha_k|=1$ \,and
\begin{equation*}
  \frac{\alpha_k}{z_k-z_0}=\frac{1}{|z_k-z_0|}.
\end{equation*}
Given \,$c\geq 0$ \,and \,$k\in\N$, we define the function \,$\varphi_{c,n}\in H(\Omega)$ \,as follows:
\begin{equation*}
  \varphi_{c,n}(z)=\begin{cases}
                     e^{\frac{c\alpha_k}{z-z_0}} & \text{if $n=n_k$ for some } k\in\N \\
                     e^{\frac{c}{z-z_0}}         & \text{if } n\notin\{n_k:k\in\N\}.
                   \end{cases}
\end{equation*}
The vector space \,$W$ \,is defined as
\begin{equation*}
  W = {\rm span} \left\{(\varphi_{c,n}\cdot f_n) : c\in[0,+\infty)\right\}
\end{equation*}
and satisfies that \,${\bf f}=(f_n)\in W$. As before, we shall prove that \,$\dim(W)=\mathfrak{c}$ \,and that \,$W\setminus\{0\}$ \,is contained in \,$\mathcal{S}_{uc}\setminus\mathcal{S}_u$.

\vskip 3pt

Let \,$m\in\N$, $\lambda_1,\ldots,\lambda_m\in\C\setminus\{0\}$, and \,$0\leq c_1<\cdots<c_m$. For each compact subset \,$K\subset\Omega$ \,there is \,$R>0$ \,such that \,$K\cap D(z_0,R)=\varnothing$. Note that if \,$c\geq 0$ \,and \,$k\in\N$, then
\begin{equation*}
  \sup_{|z-z_0|\geq R}\left|e^{\frac{c\alpha_k}{z-z_0}}\right|
  = \sup_{|w|\leq 1/R}\left|e^{c\alpha_k w}\right|
  = \sup_{|w|\leq 1/R}\left|e^{cw}\right|
  = e^{c/R}.
\end{equation*}
Therefore,
\begin{align*}
  \sup_{z\in K}\left|\sum_{j=1}^{m}\lambda_j \varphi_{c_j,n}(z)\cdot f_n(z)\right|
  &\leq \sum_{j=1}^{m}|\lambda_j| \cdot \sup_{|z-z_0|\geq R}\left|\varphi_{c_j,n}(z)\right|
  \cdot \sup_{z\in K}|f_n(z)| \\
  &= \sum_{j=1}^{m}|\lambda_j| \cdot e^{c_j/R} \cdot \sup_{z\in K}|f_n(z)| \longrightarrow 0
  \text{ as } n\to\infty
\end{align*}
because \,$(f_n)$ \,converges to zero uniformly on compacta. This proves that \,$\left(\sum_{j=1}^{m}\lambda_j \varphi_{c_j,n}\cdot f_n\right)$ \,belongs to \,$\mathcal{S}_{uc}$.

\vskip 3pt

Now we will prove that \,$\left(\sum_{j=1}^{m}\lambda_j \varphi_{c_j,n}\cdot f_n\right)$ \,does not converge to zero uniformly on \,$\left\{z_k:k\in\N\right\}$. Since \,$c_1<\cdots<c_m$ \,and \,$\lim_{k\to\infty}z_k=z_0$, there is \,$k_0\in\N$ \,such that
\begin{equation*}
  \left|\sum_{j=1}^{m-1}\lambda_j e^{\frac{c_j-c_m}{|z_k-z_0|}}\right| \leq \frac{|\lambda_m|}{2}
\end{equation*}
for all \,$k\geq k_0$. Then
\begin{align*}
  \sup_{k\geq k_0}\left|\sum_{j=1}^{m}\lambda_j \varphi_{c_j,n_k}(z_k)\cdot f_{n_k}(z_k)\right|
  &\geq \sup_{k\geq k_0}
   \varepsilon\left|\sum_{j=1}^{m}\lambda_j e^{\frac{c_j \alpha_k}{z_k-z_0}}\right|
  = \sup_{k\geq k_0} \varepsilon\left|\sum_{j=1}^{m}\lambda_j e^{\frac{c_j}{|z_k-z_0|}}\right| \\
  &= \sup_{k\geq k_0} \varepsilon \cdot e^{\frac{c_m}{|z_k-z_0|}}
  \left|\lambda_m + \sum_{j=1}^{m-1}\lambda_j e^{\frac{c_j-c_m}{|z_k-z_0|}}\right| \\
  &\geq \sup_{k\geq k_0} \varepsilon \cdot e^{\frac{c_m}{|z_k-z_0|}} \cdot \frac{|\lambda_m|}{2}
  \geq \varepsilon \cdot \frac{|\lambda_m|}{2} > 0.
\end{align*}
This proves that \,$\left(\sum_{j=1}^{m}\lambda_j \varphi_{c_j,n}\cdot f_n\right)$ \,does not converge to zero uniformly on \,$\left\{z_k:k\geq k_0\right\}$. In particular, \,$\left(\sum_{j=1}^{m}\lambda_j \varphi_{c_j,n}\cdot f_n\right)$ \,cannot be the zero sequence, which in turn implies that the set \,$\left\{(\varphi_{c,n}\cdot f_n) : c\in[0,+\infty)\right\}$ \,is linearly independent and so \,$\dim(W)=\mathfrak{c}$. That completes the proof in the case that \,$\left\{z_k:k\in\N\right\}$ \,is bounded.
\end{proof}

Finally, we improve Theorems \ref{Sp-Suc lineabilidad puntual} and \ref{Suc-Su lineabilidad puntual} in the sense that the subspaces inside our families of sequences can be chosen dense and with infinite amount of them.

\begin{theorem}\label{Pointwise dense-lineable}
The sets \,$\mathcal{S}_p\setminus\mathcal{S}_{uc}$ \,and \,$\mathcal{S}_{uc}\setminus\mathcal{S}_u$ \,are infinitely pointwise $\mathfrak{c}$-dense-lineable in \,$H(\Omega)^{\N}$.
\end{theorem}

\begin{proof}
Let \,$\mathcal{B}$ \,be defined as the set of all sequences \,$(g_n)$ \,of holomorphic functions on \,$\Omega$ \,with the following property: there is \,$n_0\in\N$ \,such that \,$g_n=0$ \,for every \,$n\geq n_0$. Then \,$\mathcal{B}$ \,is dense in \,$H(\Omega)^{\N}$. Indeed, let \,$\widetilde{d}$ \,denote the metric in \,$H(\Omega)^{\N}$ \,defined in Section \ref{background}. Given \,$(f_n)\subset H(\Omega)$ \,and \,$\varepsilon>0$, there exists \,$n_0\in\N$ \,such that
\begin{equation*}
  \sum_{n=n_0}^{\infty}\frac{1}{2^n} < \varepsilon.
\end{equation*}
We define \,$g_n=f_n$ \,if \,$n<n_0$ \,and \,$g_n=0$ \,if \,$n\geq n_0$. Then \,$(g_n)\in \mathcal{B}$. Moreover,
\begin{align*}
  \widetilde{d}((f_n),(g_n))
  = \sum_{n=1}^{\infty}\frac{1}{2^n} \cdot \frac{d(f_n,g_n)}{1+d(f_n,g_n)}
  = \sum_{n=n_0}^{\infty}\frac{1}{2^n} \cdot \frac{d(f_n,0)}{1+d(f_n,0)}
  \leq \sum_{n=n_0}^{\infty}\frac{1}{2^n} < \varepsilon.
\end{align*}
This proves that \,$\mathcal{B}$ \,is dense in \,$H(\Omega)^{\N}$. Since \,$\mathcal{B}$ \,is in fact a vector space, it follows that it is even dense-lineable in \,$H(\Omega)^{\N}$.

\vskip 3pt

If \,$(f_n)\in\mathcal{S}_p\setminus \mathcal{S}_{uc}$ \,and \,$(g_n)\in\mathcal{B}$, then there is \,$n_0\in\N$ \,such that \,$f_n+g_n=f_n$ \,for all \,$n\geq n_0$. Therefore, \,$(f_n+g_n)\in \mathcal{S}_p\setminus \mathcal{S}_{uc}$. In other words, we have
\begin{equation*}
  \left(\mathcal{S}_p \setminus \mathcal{S}_{uc}\right) + \mathcal{B}
  \subset \mathcal{S}_p\setminus \mathcal{S}_{uc}.
\end{equation*}
Moreover, every sequence \,$(g_n)$ \,in \,$\mathcal{B}$ \,converges to zero uniformly on \,$\Omega$, so \,$\left(\mathcal{S}_p\setminus \mathcal{S}_{uc}\right)\cap\mathcal{B}=\varnothing$. Finally, the set \,$\mathcal{S}_p\setminus\mathcal{S}_{uc}$ \,is pointwise $\mathfrak{c}$-lineable by Theorem \ref{Sp-Suc lineabilidad puntual}. By Theorem \ref{A stronger than B. Pointwise}, we deduce that \,$\mathcal{S}_p\setminus\mathcal{S}_{uc}$ \,is infinitely pointwise $\mathfrak{c}$-dense-lineable. The proof for the set \,$\mathcal{S}_{uc}\setminus\mathcal{S}_u$ \,is analogous, applying Theorem \ref{Suc-Su lineabilidad puntual} instead of \ref{Sp-Suc lineabilidad puntual}.
\end{proof}

\section{Banach spaces inside \,$\mathcal{S}_{p}\setminus \mathcal{S}_{uc}$ \,and \,$\mathcal{S}_{uc}\setminus\mathcal{S}_u$}

\quad Finally, we shall prove that both families \,$\mathcal{S}_{p} \setminus \mathcal{S}_{uc}$ \,and \,$\mathcal{S}_{uc} \setminus \mathcal{S}_u$ \,enjoy a property that is close to spaceability. Specifically, two large Banach spaces can be constructed so as to be respectively contained, except for zero, in the mentioned families. Prior to stating our theorem, we recall the notion of weighted Banach space \,$H_v(\C)$ \,of entire functions associated to a continuous function \,$v:\C \to (0,+\infty)$, that is defined as
\begin{equation*}
  H_v(\C ) := \left\{\varphi \in H(\C) : \, \sup_{z\in\C} v(z) \cdot |\varphi(z)|< +\infty\right\}
\end{equation*}
(see, e.g., the surveys \cite{Bonet2003} and \cite{Bonet2022} and the references contained in them). It is known that \,$H_v(\C)$ \,is a Banach space when it is endowed with the norm
\begin{equation*}
  \Vert \varphi\Vert _v := \sup_{z\in \C}v(z) \cdot |\varphi(z)|
\end{equation*}
and that convergence with respect to \,$\left\Vert \cdot \right\Vert _v$ \,implies uniform convergence on compacta in \,$\C$.

\begin{theorem}\label{Thm: pseudospaceability}
There exist two Banach spaces \,$X_1,X_2 \subset H(\Omega)^{\N}$ \,satisfying the following properties:
\begin{enumerate}
  \item[\rm (a)] Both \,$X_1$ \,and \,$X_2$ \,are infinite dimensional.
  \item[\rm (b)] The norm topology of each \,$X_i$ (with \,$i=1,2$) is stronger than the one inherited from \,$H(\Omega)^{\N}$.
  \item[\rm (c)] $X_1 \setminus \{0\} \subset \mathcal{S}_{p} \setminus \mathcal{S}_{uc}$ \,and \,$X_2 \setminus \{0\} \subset \mathcal{S}_{uc} \setminus \mathcal{S}_u$.
\end{enumerate}
\end{theorem}

\begin{proof}
According to the proof of Theorem \ref{Sp-Suc algebrable}, we may suppose that \,$0\in\Omega$. Let \,$R=\sup\left\{x\in\R : x\geq 0 \text{ and } [0,x]\subset \Omega\right\}$ \,and let \,$(t_n)$ \,be any increasing sequence of strictly positive numbers such that \,$\lim_{n\to\infty}t_n=R$. Setting
\begin{equation}\label{Eq: mayoración de los sn}
  s_n:=\frac{t_1}{2n^2}
\end{equation}
for each \,$n\in\N$, we have that \,$(s_n)$ \,is a strictly decreasing sequence of positive numbers such that \,$s_1<t_1$ \,and \,$\lim_{n\to\infty}s_n=0$. As it was shown in the proof of Theorem \ref{Sp-Suc algebrable}, there exists a sequence of holomorphic functions \,${\bf f}=(f_n)\in \mathcal{S}_{p} \setminus \mathcal{S}_{uc}$ \,such that
\begin{equation*}
  |f_n(s_{n+1}) - e^n| < e^{-n^2}
\end{equation*}
(take \,$c=1$ \,in the construction therein). As a consequence, we obtain that
\begin{equation} \label{Eq: mayoración de fn(sn)}
  |f_n(s_{n+1})| > e^n - 1
\end{equation}
for all \,$n\in\N$.

We consider the Banach space \,$H_v(\C)$ \,corresponding to \,$v(z):= e^{-|z|}$ \,and define
\begin{equation*}
  X_1 := \big\{\varphi \cdot {\bf f} = (\varphi \cdot f_n) : \, \varphi \in H_v (\C) \big\}.
\end{equation*}
Plainly, \,$X_1$ \,is a vector space and \,$X_1 \subset H(\Omega)^{\N}$.

\vskip 3pt

Since \,${\bf f}\neq 0$, there is \,$m\in\N$ \,with \,$f_m\neq 0$. Hence, by continuity, there is a nonempty open set \,$U \subset \Omega$ \,such that \,$f_m(z) \neq 0$ \,for all \,$z \in U$. Now, if \,${\bf g} = (g_n) \in X_1$ \,and \,$\varphi_1 \cdot {\bf f} = {\bf g} = \varphi_2 \cdot {\bf f}$, where \,$\varphi_1, \varphi_2 \in H_v(\C)$, then \,$\varphi_1 \cdot f_m = \varphi_2 \cdot f_m$, so \,$\varphi_1 = \varphi_2$ \,on \,$U$. It follows from the identity principle that \,$\varphi_1 = \varphi_2$ \,on the whole \,$\C$. Therefore, the mapping
\begin{equation*}
  {\bf g} = \varphi \cdot {\bf f} \in X_1 \longmapsto \|{\bf g}\|_{X_1} := \|\varphi \|_v
\end{equation*}
is well defined and it is in fact a norm on \,$X_1$. Since \,$\|\cdot\|_v$ \,is a complete norm on \,$H_v(\C)$, it follows that \,$\|\cdot \|_{X_1}$ \,is a complete norm on \,$X_1$, so as to make it a Banach space.

\vskip 3pt

Now, let us prove (a), (b), and (c) for \,$X_1$:

\vskip 3pt

\noindent (a) That \,dim$(X_1)= \infty$ \,follows from the fact that the monomials \,$\varphi_k(z)=z^k$ (with \,$k\in\N$) are linearly independent and belong to \,$H_v(\C)$. This implies that the functions \,$z^k \cdot f_m(z)$ (with \,$k\in\N$) \,are also linearly independent (recall that \,$f_m(z)\neq 0$ \,for every $z$ \,in a nonempty open subset \,$U\subset\Omega$). Hence, trivially, the set of sequences \,$\left\{z^k \cdot {\bf f} : k\in\N\right\}$ \,is linearly independent and contained in \,$X_1$, so \,$X_1$ \,is infinite dimensional.

\vskip 3pt

\noindent (b) We must show that if \,$(\psi_k)$ \,is a sequence in \,$H_v(\C)$ \,and \,$\lim_{k\to\infty}\|\psi_k\cdot{\bf f}\|_{X_1}=0$, then also \,$\lim_{k\to\infty}\psi_k\cdot{\bf f}=0$ \,with respect to the product topology on \,$H(\Omega)^{\N}$, that is, \,$\lim_{k\to\infty}\psi_k\cdot f_n=0$ \,uniformly on every compact subset of \,$\C$ \,for each \,$n\in\N$. For a prescribed \,$n\in\N$ \,and a compact subset \,$K\subset\Omega$, let us define
\begin{equation*}
  \alpha := e^{\sup_{z \in K} |z|} \cdot \sup_{z \in K} |f_n(z)| \in [0,+\infty).
\end{equation*}
Then
\begin{align*}
  \sup_{z\in K} |\psi_k(z) \cdot f_n(z)|
  &\leq \alpha \cdot \sup_{z \in K}\left(e^{-|z|} \cdot |\psi_k(z)|\right)
  \leq \alpha \cdot \sup_{z \in \C} \left(e^{-|z|} \cdot |\psi_k(z)|\right) \\
  &= \alpha \cdot \|\psi_k \|_v
  = \alpha \cdot \|\psi_k \cdot {\bf f} \|_{X_1} \longrightarrow 0  \text{ \,as } k\to\infty.
\end{align*}
This had to be shown.

\vskip 3pt

\noindent (c) Finally, suppose that \,${\bf g}\in X_1\setminus \{0\}$. Then there is \,$\varphi\in H_v(\C)$ \,such that \,$\varphi\neq 0$ \,and \,${\bf g}=\varphi \cdot {\bf f}$. On the one hand, given \,$z_0\in\Omega$, we have \,$\lim_{n\to \infty}f_n(z_0)=0$ \,and so \,$\lim_{n\to\infty}g_n(z_0) = \lim_{n\to \infty}\varphi(z_0) \cdot f_n(z_0) = 0$. In other words, \,${\bf g}\in\mathcal{S}_p$. On the other hand, since \,$\varphi \neq 0$, we can consider the multiplicity \,$\mu\in\N\cup\{0\}$ \,of \,$0$ \,as a zero of \,$\varphi$. There is an entire function \,$\phi$ \,such that \,$\phi(0)\neq 0$ \,and \,$\varphi(z)=\phi(z)\cdot z^{\mu}$ \,for all \,$z\in\C$ (see \cite[Theorem 10.18]{Rudin}). Then there exists a constant \,$c \in (0,+\infty)$ \,such that
\begin{equation*}
  |\varphi (z)| \geq c \cdot |z|^{\mu}
\end{equation*}
on some closed disc \,$\overline{D}(0,\delta)\subset \Omega$. Since \,$\lim_{n\to\infty}s_n=0$, there is \,$n_0\in\N$ \,such that \,$s_n\in\overline{D}(0,\delta)$ \,whenever \,$n \geq n_0$. Now, it follows from \eqref{Eq: mayoración de los sn} and \eqref{Eq: mayoración de fn(sn)} that for all \,$n \geq n_0$ \,we have
\begin{align*}
  \sup_{z\in\overline{D}(0,\delta)} |g_n(z)|
  &= \sup_{z\in\overline{D}(0,\delta)} |\varphi(z)| \cdot |f_n(z)|
  \geq |\varphi(s_{n+1})| \cdot |f_n(s_{n+1})| \\
  &\geq c \cdot |s_{n+1}|^{\mu} \cdot (e^n - 1)
  = \frac{c \cdot t_1^\mu \cdot (e^n - 1)}{2^{\mu} \cdot n^{2 \mu}} \longrightarrow +\infty
  \text{ \ as } \, n \to \infty.
\end{align*}
Consequently, ${\bf g} \not\in \mathcal{S}_{uc}$, as desired.

\vskip 3pt

The construction of the Banach space \,$X_2$ \,is analogous, with the sole exception that the sequence \,${\bf f}=(f_n)$ \,is extracted from \,$\mathcal{S}_{uc} \setminus \mathcal{S}_{u}$ \,instead of \,$\mathcal{S}_{p} \setminus \mathcal{S}_{uc}$. Then the properties (a) and (b) can be proved exactly in the same way. Thus, it only remains to show (c) in this case, that is, \,$X_2\setminus \{0\} \subset \mathcal{S}_{uc} \setminus \mathcal{S}_{u}$.

\vskip 3pt

To this end, let \,${\bf g} = (g_n) = (\varphi \cdot f_n) \in X_2\setminus\{0\}$. On the one hand, given a compact set \,$K \subset \Omega$, the function \,$\varphi$ \,is bounded on it, so that there is a constant \,$\beta \in (0,+\infty)$ \,with \,$|\varphi(z)| \leq \beta$ \,for all \,$z\in K$. Hence,
\begin{equation*}
  \sup_{z\in K} |g_n(z)| \leq \beta \cdot \sup_{z\in K} |f_n(z)| \longrightarrow 0
  \text{ \,as } n \to\infty
\end{equation*}
because \,$(f_n) \in \mathcal{S}_{uc}$. Thus, \,${\bf g} \in \mathcal{S}_{uc}$. It remains to prove that \,${\bf g} \not \in \mathcal{S}_{u}$. At this point, we distinguish two cases:
\begin{itemize}
\item $\Omega=\C$. Then simply select \,${\bf f} = (z/n)$ \,to define \,$X_2$. Since \,${\bf g}\neq 0$, we have that \,$\varphi \neq 0$. Then \,$z\cdot \varphi(z)$ \,is not a constant function and so it cannot be bounded on \,$\C$ \,by the Liouville's theorem (see, e.g., \cite[p.~122]{Ahlfors}). Therefore,
    \begin{equation*}
      \sup_{z\in \C} |g_n(z)| = \frac{1}{n} \cdot \sup_{z \in \C} |z \cdot \varphi (z)| = +\infty
    \end{equation*}
    for all \,$n\in\N$. Thus, ${\bf g} \not \in \mathcal{S}_{u}$.

\item $\Omega \neq \C$. Following the proof of Theorem \ref{Suc-Su algebrable}, we can assume after a translation and a rotation that \,$0 \in \Omega$ \,and that
    \begin{equation*}
      R:= \sup\left\{x\in\R : x\geq 0 \text{ and } [0,x]\subset\Omega\right\} \in \partial\Omega.
    \end{equation*}
    Let \,$0<\gamma<\min\{1,R\}$. For each \,$n\in\N$, let
    \begin{equation*}
      K_n := \left\{z\in\C : |z|\leq n \, \text{ and } \, {\rm dist}(z,\C\setminus\Omega)\geq \frac{1}{n}\right\}
    \end{equation*}
    and
    \begin{equation*}
      t_n := R-\frac{\gamma}{n^2}.
    \end{equation*}
    Then \,$(t_n)$ \,is a strictly increasing sequence of positive numbers in \,$\Omega$ \,such that \,$\lim_{n\to\infty}t_n=R$. In addition, \,$R-t_n<1/n$, so \,$t_n\notin K_n$ \,for each \,$n\in\N$. Hence, \,$(K_n)$ \,and \,$(t_n)$ \,satisfy the properties required in the proof of Theorem \ref{Suc-Su algebrable}. Therefore, the mentioned proof indicates that the sequence \,$(f_n)$ \,can be selected so that \,$|f_n(t_n)-e^n|<e^{-n^2}$ \,for all \,$n\in\N$. This implies
    \begin{equation*}
      |f_n(t_n)| > e^n - 1
    \end{equation*}
    for all \,$n\in\N$. Letting \,$\mu \in \N \cup \{0\}$ \,be the multiplicity of the point \,$R$ \,as a zero of \,$\varphi$, there exists \,$c \in (0,+\infty)$ \,such that
    \begin{equation*}
      |\varphi(z)| \geq c \cdot |z-R|^\mu
    \end{equation*}
    for all \,$z$ \,on some neighbourhood \,$V$ \,of \,$R$. If \,$n_0 \in \N$ \,is chosen such that \,$t_n \in V$ \,for all \,$n\geq n_0$, then for \,$n\geq n_0$ \,we get that
    \begin{align*}
      \sup_{z\in \Omega} |g_n(z)| &\geq |\varphi (t_n)| \cdot |f_n(t_n)|
      \geq c \cdot (R - t_n)^{\mu} \cdot (e^n - 1) \\
      &= \frac{c \cdot \gamma^\mu \cdot (e^n - 1)}{n^{2 \mu}} \longrightarrow +\infty
       \text{ \,as } n \to \infty.
    \end{align*}
    Consequently, ${\bf g} \not\in \mathcal{S}_{u}$, as desired.
\end{itemize}
The proof is now finished.
\end{proof}

\section{Final remarks and questions}

\begin{question}
Regarding Theorem \ref{Thm: pseudospaceability}, it would be interesting to know whether or not the families \,$\mathcal{S}_{p} \setminus \mathcal{S}_{uc}$ \,and \,$\mathcal{S}_{uc} \setminus \mathcal{S}_u$ \,are spaceable \textit{in \,$H(\Omega)^{\N}$}, that is, in their natural Fr\'echet space. And, similarly to the results in section 4, is it possible to fix \,${\bf f}$ \,in each of these families so that it is contained in the corresponding closed subspaces?
\end{question}

%\vskip 3pt

\begin{remark}
A question related to the topic of this paper is that of the {\it topological size} of the sets under analysis. The problem reveals to be clearer in the case of \,$\mathcal{S}_{uc} \setminus \mathcal{S}_u$, thanks to the Baire category theorem. In fact, we can say that, in a topological sense, the sequences of holomorphic functions converging compactly to zero are generically not uniformly convergent. However, to see this, we meet a serious handicap: $\mathcal{S}_{uc}$ \,is {\it not} closed as a subset of \,$H(\Omega)^{\N}$. Indeed, the sequence
\begin{equation*}
  {\bf f}_i = (\underbrace{1,1,\dots ,1}_{i \text{ times}},0,0,0, \dots )
\end{equation*}
belongs to \,$\mathcal{S}_{uc}$ \,for each \,$i\in\N$ \,and \,$\lim_{i\to\infty}{\bf f}_i=(1,1,1,\dots) \not\in \mathcal{S}_{uc}$. Hence, \,$\mathcal{S}_{uc}$ \,is not a complete metric space under the inherited distance \,$\widetilde{d}$ \,defined in Section \ref{background} and, consequently, Baire's theorem is not at our disposal. Nevertheless, it is possible to replace \,$\widetilde{d}$ \,with a new distance making \,$\mathcal{S}_{uc}$ \,complete. For this, notice that if \,$K \subset \Omega$ \,is a compact subset and \,${\bf f} = (f_n) \in \mathcal{S}_{uc}$, then
\begin{equation*}
  \|{\bf f}\|_K := \sup_{n \in \N}\sup_{z \in K} |f_n(z)|
\end{equation*}
is finite (use the continuity of each \,$f_n$ \,together with the uniform convergence on \,$K$). Choose any exhaustive sequence \,$(K_j)$ \,of compact subsets of \,$\Omega$ \,as in \cite[Theorem 13.3]{Rudin}. Defining
\begin{equation*}
  D({\bf f},{\bf g}) := \sum_{j=1}^{\infty}
  \frac{1}{2^j} \cdot \frac{\|{\bf f} - {\bf g}\|_{K_j}}{1+ \|{\bf f} - {\bf g}\|_{K_j}},
\end{equation*}
then it is an easy (long, but standard) exercise to prove that \,$D$ \,is a complete distance on \,$\mathcal{S}_{uc}$ \,(in fact, \,$D$ \,defines a topology that is finer than the inherited from \,$H(\Omega )^{\N}$). Therefore, \,$(\mathcal{S}_{uc},D)$ \,is a Baire space.

\vskip 3pt

We now can assert that the sequences converging compactly to zero do not generically converge uniformly. To be more precise, the set \,$\mathcal{S}_{uc} \setminus \mathcal{S}_u$ \,is {\it residual in \,$(\mathcal{S}_{uc},D)$}, that is, its complement \,$\mathcal{S}_u$ \,is of first category or, that is the same, $\mathcal{S}_u$ \,is contained in some \,$F_\sigma$-set (i.e.~a countable union of closed sets) with empty interior. Let us show why this is so. By the Baire category theorem, it is enough to prove the existence of countably many closed subsets \,$\mathcal{A}_k \subset \mathcal{S}_{uc}$ \,such that \,$\mathcal{S}_u \subset \displaystyle{\bigcup_{k \in \N} \mathcal{A}_k}$ \,and \,$(\mathcal{A}_k)^{\circ} = \varnothing$ \,for all \,$k\in\N$. For this, let us define
\begin{equation*}
  \mathcal{A}_k := \{ {\bf f} = (f_n) \in \mathcal{S}_{uc} : \, |f_k(z)| \leq 1
  \text{ for all } n \geq k \text{ and all }  z \in \Omega \}.
\end{equation*}
Then the inclusion \,$\mathcal{S}_u \subset \displaystyle{\bigcup_{k \in \N} \mathcal{A}_k}$ \,follows directly from the uniform convergence to \,$0$. Since \,$(\mathcal{S}_{uc},D)$ \,carries a topology that is stronger than the product topology and compact convergence is stronger than pointwise convergence, it follows that, for each \,$(n,z) \in \N \times \Omega$, the mapping
\begin{equation*}
  \Phi_{n,z} : \, {\bf f} \in \mathcal{S}_{uc} \longmapsto |f_n(z)| \in \R
\end{equation*}
is continuous for the metric \,$D$. But \,$\mathcal{A}_k = \displaystyle{\bigcap_{z \in \Omega \atop n \ge k} \Phi_{n,z}^{-1}([0,1])}$, which shows that each of these sets is closed in \,$(\mathcal{S}_{uc},D)$.

\vskip 3pt

Now, assume, by way of contradiction, that \,$(\mathcal{A}_k)^\circ \neq \varnothing$ \,for some \,$k\in\N$. Then there would exist a basic open subset \,$U = U({\bf f},K,\varepsilon)$ \,with \,$U\subset\mathcal{A}_k$, where \,${\bf f} \in \mathcal{S}_{uc}$, $K$ \,is a compact subset of \,$\Omega$, $0<\varepsilon<1$, and
\begin{equation*}
  \phantom{aaaaa} U := \big\{ {\bf g}=(g_n) \in \mathcal{S}_{uc} : |g_n(z) - f_n(z)| < \varepsilon \, \text{ for all } n \in \N \text{ and all } z \in K \big\}.
\end{equation*}
Let \,$(K_n)$ \, be any exhaustive sequence of compact subsets of \,$\Omega$ \,with the properties given in \cite[Theorem 13.3]{Rudin}. Then there is \,$n_0\in\N$ \,such that \,$K\subset K_n$ \,for all \,$n\geq n_0$. We define \,$g_n=f_n$ \,for each \,$n<n_0$.

\vskip 3pt

Let \,$E$ \,be a set defined by choosing one point at each connected component of \,$\C_\infty\setminus \Omega$. Hence, $E$ \,has a point at every connected component of \,$\C_\infty\setminus K_n$ \,for all \,$n\in\N$. For each \,$n\geq n_0$, we select a point \,$z_n \in \Omega \setminus K_n$ \,and a real number \,$r_n>0$ \,with the following properties:
\begin{itemize}
  \item $D(z_n,r_n)\subset\Omega$.
  \item $\big(K_n+D(0,r_n)\big)\subset \Omega$.
  \item $\big(K_n+D(0,r_n)\big)\cap D(z_n,r_n)=\varnothing$.
\end{itemize}
Set \,$W_n := \big(K_n+D(0,r_n)\big)\cup D(z_n,r_n)$ \,and \,$L_n := K_n \cup \{z_n\}$. Then \,$L_n$ \,is a compact subset of the open set \,$W_n$ \,and \,$E$ \,has a point at each connected component of \,$\C_{\infty}\setminus L_n$. For each \,$n\geq n_0$, define the function \,$h_n: W_n \to \C$ \,by
\begin{equation*}
  h_n(z) =
  \begin{cases}
    f_n(z) & \text{if } z \in K_n+D(0,r_n) \\
    2      & \text{if } z \in D(z_n,r_n).
  \end{cases}
\end{equation*}
At this point, the Runge approximation theorem yields the existence of a rational function \,$g_n$ \,whose poles only lie in \,$E$ \,such that \,$|g_n(z)-h_n(z)|<\frac{\varepsilon}{n}$ \,for all \,$z\in L_n$. In particular, $g_n \in H(\Omega)$. Moreover, we have
\begin{equation}\label{Residual |g-f|<epsilon/n}
  |g_n(z)-f_n(z)| < \frac{\varepsilon}{n} \, \text{ for all } z\in K_n \, \text{ and all } n\geq n_0
\end{equation}
and
\begin{equation}\label{Residual |g-2|<1}
  |g_n(z_n)-2| < \varepsilon < 1 \, \text{ for all } n\geq n_0.
\end{equation}
On the one hand, from \eqref{Residual |g-f|<epsilon/n} and the facts that \,${\bf f} \in \mathcal{S}_{uc}$ \,and that every compact subset of \,$\Omega$ \,is eventually contained in all \,$K_n$'s, it follows easily that \,${\bf g} := (g_n) \in \mathcal{S}_{uc}$ \,and even \,${\bf g} \in U$. On the other hand, we derive from \eqref{Residual |g-2|<1} that \,$|g_n(z_n)|>1$ \,for all \,$n\geq n_0$, which implies that \,${\bf g} \not\in \mathcal{A}_k$. Therefore, \,$U \not\subset \mathcal{A}_k$ \,and this is the sought-after contradiction.
\end{remark}

\begin{question}
In connection with the preceding remark, we want to raise the following question: What is the topological structure as well as the topological size (in some precise sense) of \,$\mathcal{S}_p \setminus \mathcal{S}_{uc}$ \,as a subset of \,$H(\Omega)^{\N}$?
\end{question}

\vskip 3pt

\noindent {\bf Acknowledgments.} This paper was started during a stay of the third author at the Instituto de Matem\'aticas de la Universidad de Sevilla (IMUS). The third author wants to thank this institution for its support and hospitality.

\vskip 3pt

\noindent {\bf Authors contribution.} All four authors have contributed equally to this work.

\end{document}